\documentclass[reqno]{amsart}
\usepackage{amssymb}
\usepackage{hyperref}

\begin{document}
\title[Certain unified integration formulas associated ....]
{Certain unified integration formulas associated with generalized $k$-Bessel function}

\author[Gauhar Rahman, K. S. Nisar, Shahid Mubeen, Muhammad Arshad]
{ G. Rahman, K. S. Nisar, S. Mubeen, M. Arshad}  
\address{Department of Mathematics University of Sargodha, Sargodha, Pakistan}
\address{Gauhar Rahman\newline
 Department of Mathematics, International Islamic University,
Islamabad, Pakistan}
\email{gauhar55uom@gmail.com}
\address{Kottakkaran Sooppy Nisar\newline Department of Mathematics, College of Arts and
Science, Prince Sattam bin Abdulaziz University, Wadi Al dawaser, Riyadh
region 11991, Saudi Arabia}
\email{ksnisar1@gmail.com, n.sooppy@psau.edu.sa}

\address{Shahid Mubeen \newline
Department of Mathematics, University of Sargodha, Sargodha, Pakistan}
\email{smjhanda@gmail.com}
\address{ Muhammad Arshad \newline
Department of Mathematics, International Islamic University,
Islamabad, Pakistan}
\email{marshad$_{-}$zia@yahoo.com}

\thanks{Submitted .}

\begin{abstract}
 Our purpose in this present paper is to investigate generalized integration formulas containing the generalized $k$-Bessel function $W_{v,c}^{k}(z)$  to obtain the results in representation of Wright-type function. Also, we establish certain special cases of our main result.   
 \end{abstract}


\vskip 3mm

\maketitle
\numberwithin{equation}{section}
\newtheorem{theorem}{Theorem}[section]
\newtheorem{lemma}[theorem]{Lemma}
\newtheorem{proposition}[theorem]{Proposition}
\newtheorem{corollary}[theorem]{Corollary}
\newtheorem*{remark}{Remark}

\section{introduction}
The generalized $k$-Bessel function  defined in \cite{Mondal2016} as:
\begin{eqnarray}\label{1}
W_{v, c}^{k}(z)=\sum\limits_{n=0}^{\infty}\frac{(-c)^n}{\Gamma_k(nk+v+k)n!}(\frac{z}{2})^{2n+\frac{v}{k}},
\end{eqnarray}
where $k>0$, $v>-1$, and $c\in\mathbb{R}$ and $\Gamma_k(z)$ is the $k$-gamma function defined in \cite{Diaz2007} as:
 \begin{eqnarray}
 \Gamma_k(z)=\int\limits_{0}^{\infty}t^{z-1}e^{-\frac{t^k}{k}}dt, z\in\mathbb{C}.
 \end{eqnarray}
 By inspection the following relation holds:
 \begin{eqnarray}
 \Gamma_k(z+k)=z\Gamma_k(z)
 \end{eqnarray}
 and
 \begin{eqnarray}\label{2}
 \Gamma_k(z)=k^{\frac{z}{k}-1}\Gamma(\frac{z}{k}).
 \end{eqnarray}
 In the same paper, the researchers also defined Pochhammer $k$-symbols which is defined as:
 \begin{eqnarray*}
 (x)_{n,k}= x(x+k)\cdots(x+(n-1)k),  n\neq0, n\in\mathbb{N}, (x)_{0,k}=1.
 \end{eqnarray*}
 The relation between Pochhammer $k$-symbols and $k$-gamma function is defined as:
 \begin{eqnarray*}
 (x)_{n,k}=\frac{\Gamma_k(x+nk)}{\Gamma_k(x)}.
 \end{eqnarray*}
 If $k\rightarrow 1$ and $c=1$, then the generalized $k$-Bessel function defined in (\ref{9}) reduces to the well known classical Bessel function $J_v$ defined in \cite{Erdelyi1953}. For further detail about $k$-Bessel function and its properties (see \cite{Gehlot2014}-\cite{Gehlot2016}). \\
The generalized hypergeometric function $_pF_{q}(z)$ is defined in \cite{Erd1953} as:
$$_pF_{q}(z)=\quad_pF_{q,}
                                 \left[
                                   \begin{array}{ccc}
                                     (\alpha_1),(\alpha_2),\cdots(\alpha_p) &  &  \\
                                      &  & ;z \\
                                     (\beta_1),(\beta_2),\cdots(\beta_q) &  &  \\
                                   \end{array}
                                 \right]$$
\begin{eqnarray}\label{3}
=\sum\limits_{n=0}^{\infty}\frac{(\alpha_1)_{n}(\alpha_2)_{n}\cdots(\alpha_p)_{n}}
{(\beta_1)_{n}(\beta_2)_{n}\cdots(\beta_q)_{n}}\frac{z^{n}}{n!},
\end{eqnarray}
 where $\alpha_i, \beta_j\in\mathbb{C}$; $i=1,2,\cdots,p$, $j=1,2,\cdots,q$ and $b_j\neq0, -1, -2,\cdots$
 and $(z)_{n}$ is the Pochhammer symbols. The gamma function is defined  as:
 \begin{eqnarray}
 \Gamma(\mu)=\int\limits_{0}^{\infty}t^{\mu-1}e^{-t}dt, \mu\in\mathbb{C},
 \end{eqnarray}
 \begin{eqnarray}\label{4}
 \Gamma(z+n)=z\Gamma(z), z\in\mathbb{C},
\end{eqnarray}
and beta function is defined as:
\begin{eqnarray}\label{5}
B(x,y)=\int\limits_{0}^{1}t^{x-1}(1-t)^{y-1}dt.
\end{eqnarray}
The Wright type hypergeometric function  is defined  (see \cite{Wright1935}-\cite{Wrigt1935}) by the following series as:
$$_p\Psi_{q}(z)=\quad_p\Psi_{q}
                                 \left[
                                   \begin{array}{ccc}
                                     (\alpha_i, A_i)_{1,p} &  &  \\
                                      &  & ;z \\
                                     (\beta_j, B_j)_{1,q} &  &  \\
                                   \end{array}
                                 \right]$$
\begin{eqnarray}\label{6}
&=&\sum\limits_{n=0}^{\infty}\frac{\Gamma(\alpha_{1}+A_{1}n)\cdots\Gamma(\alpha_{p}+A_{p}n)}{\Gamma(\beta_{1}+B_{1}n)\cdots
\Gamma(\beta_{q}+B_{q}n)}\frac{z^{n}}{n!}
\end{eqnarray}
where $\beta_{r}$ and $\mu_{s}$  are real positive numbers such that
\begin{eqnarray*}
1+\sum\limits_{s=1}^{q}\beta_{s}-\sum\limits_{r=1}^{p}\alpha_{r}>0.
\end{eqnarray*}

Equation (\ref{11}) differs from the generalized hypergeometric function $_{p}F_{q}(z)$ defined (\ref{10})  only by a constant multiplier. The generalized hypergeometric function $_{p}F_{q}(z)$ is a special case of $_p\Psi_{q}(z)$ for $A_i=B_j=1$, where $i=1,2,\cdots,p$ and $j=1,2,\cdots,q$:
\begin{multline}\label{7}
\frac{1}{\prod\limits_{j=1}^{q}\Gamma(\beta_j)}\quad_pF_{q}
                                 \left[
                                   \begin{array}{ccc}
                                     (\alpha_1),\cdots(\alpha_p) &  &  \\
                                      &  & ;z \\
                                     (\beta_1),\cdots(\beta_q) &  &  \\
                                   \end{array}
                                 \right]=
\frac{1}{\prod\limits_{i=1}^{p}\Gamma(\alpha_i)}\quad_p\Psi_{q}
                                 \left[
                                   \begin{array}{ccc}
                                     (\alpha_i, 1)_{1,p} &  &  \\
                                      &  & ;z \\
                                     (\beta_j, 1)_{1,q} &  &  \\
                                   \end{array}
                                 \right].
\end{multline}
In this paper, we define a class of integral formulas which containing the generalized $k$-Bessel function as defined in (\ref{1}). Also, we investigate some special cases as the corollaries. For this continuation of our study, we recall the following result of Lavoie and Trottier \cite{Lavoie1969}.
 \begin{multline}\label{8}
 \int\limits_{0}^{1}z^{\alpha-1}(1-z)^{2\beta-1}(1-\frac{z}{3})^{2\alpha-1}(1-\frac{z}{4})^{\beta-1}dz=\left(
                                                                                                         \begin{array}{c}
                                                                                                           \frac{2}{3} \\
                                                                                                         \end{array}
  \right)^{2\alpha}\frac{\Gamma(\alpha)\Gamma(\beta)}{\Gamma(\alpha+\beta)}\\
 \mathfrak{R}(\alpha)>0,\mathfrak{R}(\beta)>0.
 \end{multline}
 For various other investigation containing special function, the reader may refer to the recent work of researchers (see \cite{Choi2013}, \cite{Choi2014},  \cite{Menaria2016},  \cite{Nisar}, \cite{Nisara}).
 \section{\textbf{Main Result}}
 In this section, we establish two generalized integral formulas containing $k$-Bessel function defined (\ref{1}), which represented in terms of Wright-type function defined in (\ref{6}) by inserting with the suitable argument defined in (\ref{8}).
 \begin{theorem}\label{2.1}
 For $\lambda$, $\rho$, $v$, $c\in\mathbb{C}$ with $\mathfrak{R}(\frac{v}{k})>-1$, $\mathfrak{R}(\lambda+\rho)>0$, $\mathfrak{R}(\lambda+\frac{v}{k})>0$ and $z>0$, then the following result holds:
 \begin{eqnarray*}
 \int\limits_{0}^{1}z^{\lambda+\rho-1}(1-z)^{2\lambda-1}(1-\frac{z}{3})^{2(\lambda+\rho)-1}(1-\frac{z}{4})^{\lambda-1}
 W_{v,c}^{k}\left(
              \begin{array}{c}
                \frac{y\left(
                   \begin{array}{c}
                     1-\frac{z}{4} \\
                   \end{array}
                 \right)
                 \left(
                   \begin{array}{c}
                     1-z \\
                   \end{array}
                 \right)^2}{2}
              \end{array}
            \right)dz
 \end{eqnarray*}
 \begin{multline}\label{9}
 =\frac{(\frac{y}{2})^\frac{v}{k}\Gamma(\lambda+\rho)(\frac{2}{3})^{2(\lambda+\rho)}}{k^\frac{v}{k}}\\
 \times \quad_1\Psi_{2}
\left[
\begin{array}{ccc}
 (\lambda+\frac{v}{k}, 2); \\
 &   & \quad | -\frac{cy^2}{4k}\\
(\frac{v}{k}+1,1), (2\lambda+\frac{v}{k}+\rho, 2) \\
\end{array}
\right].
 \end{multline}
 \end{theorem}
 \begin{proof}
 Let $S$ be the left hand side of (\ref{2.1}) and applying (\ref{1}) to the integrand of (\ref{9}), we have
 \begin{eqnarray*}
 S&=&\int\limits_{0}^{1}z^{\lambda+\rho-1}(1-z)^{2\lambda-1}(1-\frac{z}{3})^{2(\lambda+\rho)-1}(1-\frac{z}{4})^{\lambda-1}
 \\&\times&\sum\limits_{n=0}^{\infty}\frac{(-c)^n}{\Gamma_k(nk+v+k)n!}\left(
              \begin{array}{c}
                \frac{y\left(
                   \begin{array}{c}
                     1-\frac{z}{4} \\
                   \end{array}
                 \right)
                 \left(
                   \begin{array}{c}
                     1-z \\
                   \end{array}
                 \right)^2}{2}
              \end{array}
            \right)^{2n+\frac{v}{k}}dz
 \end{eqnarray*}
 By interchanging the order of integration and summation, which is verified by the uniform convergence of the series under the given assumption of theorem \ref{2.1}, we have
 \begin{eqnarray*}
 S&=&\sum\limits_{n=0}^{\infty}\frac{(-c)^n}{\Gamma_k(nk+v+k)n!}(\frac{y}{2})^{2n+\frac{v}{k}}
 \\&\times& \int\limits_{0}^{1}z^{\lambda+\rho-1}(1-z)^{2(\lambda+\frac{v}{k}+2n)-1}(1-\frac{z}{3})^{2(\lambda+\rho)-1}(1-\frac{z}{4})^{\lambda+\frac{v}{k}+2n-1}
 dz.
 \end{eqnarray*}
 By considering the assumption given in theorem \ref{2.1}, since $\mathfrak{R}(\frac{v}{k})>0$, $\mathfrak{R}(\lambda+\frac{v}{k}+2n)>\mathfrak{R}(\lambda+\frac{v}{k})>0$, $\mathfrak{R}(\lambda+\rho)>0$, $k>0$ and applying (\ref{8}), we obtain
 \begin{eqnarray*}
 S=\sum\limits_{n=0}^{\infty}\frac{(-c)^n}{\Gamma_k(nk+v+k)n!}(\frac{y}{2})^{2n+\frac{v}{k}}(\frac{2}{3})^{2(\lambda+\rho)}
 \frac{\Gamma(\lambda+\rho)\Gamma(\lambda+\frac{v}{k}+2n)}{\Gamma(2\lambda+\rho+\frac{v}{k}+2n)}.
 \end{eqnarray*}
 Using (\ref{2}), we get
 \begin{eqnarray*}
 S=\frac{(\frac{y}{2})^\frac{v}{k}\Gamma(\lambda+\rho)(\frac{2}{3})^{2(\lambda+\rho)}}{k^\frac{v}{k}}\sum\limits_{n=0}^{\infty}\frac{(-c)^n}{\Gamma(\frac{v}{k}+1+n)n!}
 (\frac{y^{2n}}{4^nk^n})
 \frac{\Gamma(\lambda+\frac{v}{k}+2n)}{\Gamma(2\lambda+\rho+\frac{v}{k}+2n)}
 \end{eqnarray*}
 which upon using (\ref{6}), we get the required result.
 \end{proof}
  \begin{theorem}\label{2.2}
 For $\lambda$, $\rho$, $v$, $c\in\mathbb{C}$ with $\mathfrak{R}(\frac{v}{k})>-1$, $\mathfrak{R}(\lambda+\rho)>0$, $\mathfrak{R}(\lambda+\frac{v}{k})>0$ and $z>0$, then the following result holds:
 \begin{eqnarray*}
 \int\limits_{0}^{1}z^{\lambda-1}(1-z)^{2(\lambda+\rho)-1}(1-\frac{z}{3})^{2\lambda-1}(1-\frac{z}{4})^{(\lambda+\rho)-1}
 W_{v,c}^{k}\left(
              \begin{array}{c}
                \frac{yz\left(
                   \begin{array}{c}
                     1-\frac{z}{3} \\
                   \end{array}
                 \right)^2}{2}
              \end{array}
            \right)dz
 \end{eqnarray*}
 \begin{multline}\label{10}
 =\frac{(\frac{y}{2})^\frac{v}{k}\Gamma(\lambda+\rho)(\frac{2}{3})^{2(\lambda+\frac{v}{k})}}{k^\frac{v}{k}}\\
 \times \quad_1\Psi_{2}
\left[
\begin{array}{ccc}
 (\lambda+\frac{v}{k}, 2); \\
 &   & \quad | -\frac{4cy^2}{81k}\\
(\frac{v}{k}+1,1), (2\lambda+\frac{v}{k}+\rho, 2) \\
\end{array}
\right].
 \end{multline}
 \end{theorem}
\begin{proof}
 Let $\mathfrak{L}$ be the left hand side of (\ref{2.2}) and applying (\ref{1}) to the integrand of (\ref{9}), we have
 \begin{eqnarray*}
 \mathfrak{L}&=&\int\limits_{0}^{1}z^{\lambda-1}(1-z)^{2(\lambda+\rho)-1}(1-\frac{z}{3})^{2\lambda-1}(1-\frac{z}{4})^{(\lambda+\rho)-1}
 \\&\times&\sum\limits_{n=0}^{\infty}\frac{(-c)^n}{\Gamma_k(nk+v+k)n!}\left(
              \begin{array}{c}
                \frac{yz\left(
                   \begin{array}{c}
                     1-\frac{z}{3} \\
                   \end{array}
                 \right)^2}{2}
              \end{array}
            \right)^{2n+\frac{v}{k}}dz
 \end{eqnarray*}
 By interchanging the order of integration and summation, which is verified by the uniform convergence of the series under the given assumption of theorem \ref{2.2}, we have
 \begin{eqnarray*}
 \mathfrak{L}&=&\sum\limits_{n=0}^{\infty}\frac{(-c)^n}{\Gamma_k(nk+v+k)n!}(\frac{y}{2})^{2n+\frac{v}{k}}
 \\&\times& \int\limits_{0}^{1}z^{\lambda+\frac{v}{k}+2n-1}(1-z)^{2(\lambda+\rho)-1}(1-\frac{z}{3})^{2(\lambda+\frac{v}{k}+2n)-1}(1-\frac{z}{4})^{\lambda+\rho-1}
 dz.
 \end{eqnarray*}
 By considering the assumption given in theorem \ref{2.2}, since $\mathfrak{R}(\frac{v}{k})>0$, $\mathfrak{R}(\lambda+\frac{v}{k}+2n)>\mathfrak{R}(\lambda+\frac{v}{k})>0$, $\mathfrak{R}(\lambda+\rho)>0$, $k>0$ and applying (\ref{8}), we obtain
 \begin{eqnarray*}
 \mathfrak{L}=\sum\limits_{n=0}^{\infty}\frac{(-c)^n}{\Gamma_k(nk+v+k)n!}(\frac{y}{2})^{2n+\frac{v}{k}}(\frac{2}{3})^{2(\lambda+\frac{v}{k}+2n)}
 \frac{\Gamma(\lambda+\rho)\Gamma(\lambda+\frac{v}{k}+2n)}{\Gamma(2\lambda+\rho+\frac{v}{k}+2n)}.
 \end{eqnarray*}
 Using (\ref{2}), we get
 \begin{eqnarray*}
 S=\frac{(\frac{y}{2})^\frac{v}{k}\Gamma(\lambda+\rho)(\frac{2}{3})^{2(\lambda+\frac{v}{k}+2n)}}{k^\frac{v}{k}}\sum\limits_{n=0}^{\infty}\frac{(-c)^n}{\Gamma(\frac{v}{k}+1+n)n!}
 (\frac{y^{2n}}{4^nk^n})
 \frac{\Gamma(\lambda+\frac{v}{k}+2n)}{\Gamma(2\lambda+\rho+\frac{v}{k}+2n)}
 \end{eqnarray*}
 which upon using (\ref{6}), we get the required result.
 \end{proof}
\section{\textbf{Special Cases}}
In this section, we present the generalized form of classical and modified Bessel functions which are the special cases of $k$-Bessel function defined (\ref{1}). Also, we prove two corollaries which are the special cases of obtained theorems in Section 2.\\
\textbf{Case 1.} If we set $c=1$ in (\ref{1}), then we get another definition of $k$-Bessel function. We call it the classical $k$-Bessel function
\begin{eqnarray}\label{11}
J_{v}^{k}(z)=\sum\limits_{n=0}^{\infty}\frac{(-1)^n(\frac{z}{2})^{\frac{v}{k}+2n}}{\Gamma_{k}(v+nk+k)n!}
\end{eqnarray}
\textbf{Case 2.} If we set $c=-1$ in (\ref{1}), then we get another definition of $k$-Bessel function. We call it the modified $k$-Bessel function
\begin{eqnarray}\label{12}
I_{v}^{k}(z)=\sum\limits_{n=0}^{\infty}\frac{(\frac{z}{2})^{\frac{v}{k}+2n}}{\Gamma(v+nk+k)n!}
\end{eqnarray}
\begin{corollary}
Assume that the conditions of Theorem \ref{2.1} are satisfied. Then the following integral formula holds: 
\begin{eqnarray*}
 \int\limits_{0}^{1}z^{\lambda+\rho-1}(1-z)^{2\lambda-1}(1-\frac{z}{3})^{2(\lambda+\rho)-1}(1-\frac{z}{4})^{\lambda-1}
 J_{v}^{k}\left(
              \begin{array}{c}
                \frac{y\left(
                   \begin{array}{c}
                     1-\frac{z}{4} \\
                   \end{array}
                 \right)
                 \left(
                   \begin{array}{c}
                     1-z \\
                   \end{array}
                 \right)^2}{2}
              \end{array}
            \right)dz
 \end{eqnarray*}
 \begin{multline}\label{13}
 =\frac{(\frac{y}{2})^\frac{v}{k}\Gamma(\lambda+\rho)(\frac{2}{3})^{2(\lambda+\rho)}}{k^\frac{v}{k}}\\
 \times \quad_1\Psi_{2}
\left[
\begin{array}{ccc}
 (\lambda+\frac{v}{k}, 2); \\
 &   & \quad | -\frac{y^2}{4k}\\
(\frac{v}{k}+1,1), (2\lambda+\frac{v}{k}+\rho, 2) \\
\end{array}
\right].
 \end{multline}
\end{corollary}
\begin{corollary}
Assume that the conditions of Theorem \ref{2.1} are satisfied. Then the following integral formula holds:
\begin{eqnarray*}
 \int\limits_{0}^{1}z^{\lambda+\rho-1}(1-z)^{2\lambda-1}(1-\frac{z}{3})^{2(\lambda+\rho)-1}(1-\frac{z}{4})^{\lambda-1}
 J_{v}^{k}\left(
              \begin{array}{c}
                \frac{y\left(
                   \begin{array}{c}
                     1-\frac{z}{4} \\
                   \end{array}
                 \right)
                 \left(
                   \begin{array}{c}
                     1-z \\
                   \end{array}
                 \right)^2}{2}
              \end{array}
            \right)dz
 \end{eqnarray*}
 \begin{multline}\label{14}
 =\frac{(\frac{y}{2})^\frac{v}{k}\Gamma(\lambda+\rho)(\frac{2}{3})^{2(\lambda+\rho)}}{k^\frac{v}{k}}\\
 \times \quad_1\Psi_{2}
\left[
\begin{array}{ccc}
 (\lambda+\frac{v}{k}, 2); \\
 &   & \quad | -\frac{y^2}{4k}\\
(\frac{v}{k}+1,1), (2\lambda+\frac{v}{k}+\rho, 2) \\
\end{array}
\right].
 \end{multline}
\end{corollary}
\begin{corollary}
Assume that the conditions of Theorem \ref{2.2} are satisfied. Then the following integral formula holds:
\begin{eqnarray*}
 \int\limits_{0}^{1}z^{\lambda-1}(1-z)^{2(\lambda+\rho)-1}(1-\frac{z}{3})^{2\lambda-1}(1-\frac{z}{4})^{(\lambda+\rho)-1}
 J_{v}^{k}\left(
              \begin{array}{c}
                \frac{yz\left(
                   \begin{array}{c}
                     1-\frac{z}{3} \\
                   \end{array}
                 \right)^2}{2}
              \end{array}
            \right)dz
 \end{eqnarray*}
 \begin{multline}
 =\frac{(\frac{y}{2})^\frac{v}{k}\Gamma(\lambda+\rho)(\frac{2}{3})^{2(\lambda+\frac{v}{k})}}{k^\frac{v}{k}}\\
 \times \quad_1\Psi_{2}
\left[
\begin{array}{ccc}
 (\lambda+\frac{v}{k}, 2); \\
 &   & \quad | -\frac{4y^2}{81k}\\
(\frac{v}{k}+1,1), (2\lambda+\frac{v}{k}+\rho, 2) \\
\end{array}
\right].
 \end{multline}
\end{corollary}
\begin{corollary}
Assume that the conditions of Theorem \ref{2.2} are satisfied. Then the following integral formula holds:
\begin{eqnarray*}
 \int\limits_{0}^{1}z^{\lambda-1}(1-z)^{2(\lambda+\rho)-1}(1-\frac{z}{3})^{2\lambda-1}(1-\frac{z}{4})^{(\lambda+\rho)-1}
 I_{v}^{k}\left(
              \begin{array}{c}
                \frac{yz\left(
                   \begin{array}{c}
                     1-\frac{z}{3} \\
                   \end{array}
                 \right)^2}{2}
              \end{array}
            \right)dz
 \end{eqnarray*}
 \begin{multline}\label{15}
 =\frac{(\frac{y}{2})^\frac{v}{k}\Gamma(\lambda+\rho)(\frac{2}{3})^{2(\lambda+\frac{v}{k})}}{k^\frac{v}{k}}\\
 \times \quad_1\Psi_{2}
\left[
\begin{array}{ccc}
 (\lambda+\frac{v}{k}, 2); \\
 &   & \quad | -\frac{4y^2}{81k}\\
(\frac{v}{k}+1,1), (2\lambda+\frac{v}{k}+\rho, 2) \\
\end{array}
\right].
 \end{multline}
\end{corollary}
\begin{remark}
If we set $k=1$ in (\ref{11}) to (\ref{15}), then we get the well known result for case 1 (see \cite{Agarwal2014}) and some new result for the familiar function defined in \cite{Mondal2016,Baricz2008,Watson1996}. 
\end{remark}

\end{document}